\documentclass{amsart}
\usepackage{amssymb}
\usepackage{amsfonts}
\usepackage{hyperref}
\usepackage{color}

\newtheorem{theorem}{Theorem}
\theoremstyle{plain}

\newtheorem{corollary}{Corollary}

\newtheorem{lemma}{Lemma}

\newtheorem{proposition}{Proposition}
\newtheorem{remark}{Remark}

\numberwithin{equation}{section}
\input{tcilatex}

\begin{document}
\title[Absolute monotonicity involving complete elliptic integral]{%
Absolutely monotonic functions related to the asymptotic formula for the
complete elliptic integral of the first kind}
\author{Tiehong Zhao and Zhen-Hang Yang*}
\address{Tiehong Zhao\\
School of mathematics, Hangzhou Normal University, Hangzhou, Zhejiang,
China, 311121}
\email{tiehong.zhao@hznu.edu.cn}
\urladdr{https://orcid.org/0000-0002-6394-1049}
\address{Zhen-Hang Yang\\
State Grid Zhejiang Electric Power Company Research Institute, Hangzhou,
Zhejiang, China, 310014}
\email{yzhkm@163.com}
\urladdr{https://orcid.org/0000-0002-2719-4728}
\thanks{Corresponding Author: Zhen-Hang Yang*}

\begin{abstract}
Let $\mathcal{K}\left( x\right) $ be the complete elliptic integral of the
first kind and%
\begin{equation*}
\text{\ }\mathcal{G}_{p}\left( x\right) =e^{\mathcal{K}\left( \sqrt{x}%
\right) }-\frac{p}{\sqrt{1-x}}
\end{equation*}%
for $p\in \mathbb{R}$ and $x\in \left( 0,1\right) $. In this paper we find
the necessary and sufficient conditions for the functions $\pm \mathcal{G}%
_{p}^{\left( k\right) }\left( x\right) $ ($k\in \mathbb{N\cup }\left\{
0\right\} $) to be absolutely monotonic on $\left( 0,1\right) $, which
extend previous known results and yield several new functional inequalities
involving the complete elliptic integral of the first kind. More
importantly, we provide a new method to deal with those absolute
monotonicity problem by proving the monotonicity of a sequence generated by
the coefficients of the power series of $\mathcal{G}_{p}\left( x\right) $.
\end{abstract}

\maketitle

\section{Introduction}

The complete elliptic integral of the first kind $\mathcal{K}\left( r\right) 
$ is defined on $\left( 0,1\right) $ by%
\begin{equation*}
\mathcal{K}\left( r\right) =\int_{0}^{\pi /2}\frac{1}{\sqrt{1-r^{2}\sin ^{2}t%
}}dt,
\end{equation*}%
which can also be expressed by the Gaussian hypergeometric function%
\begin{equation}
\mathcal{K}\left( r\right) =\frac{\pi }{2}F\left( \frac{1}{2},\frac{1}{2}%
;1;r^{2}\right) =\frac{\pi }{2}\sum_{n=0}^{\infty }\frac{\left( 1/2\right)
_{n}^{2}}{\left( n!\right) ^{2}}r^{2n},  \label{K-hs}
\end{equation}%
where%
\begin{equation*}
F\left( a,b;c;x\right) =\sum_{n=0}^{\infty }\frac{\left( a\right) _{n}\left(
b\right) _{n}}{\left( c\right) _{n}}\frac{x^{n}}{n!}
\end{equation*}%
is the Gaussian hypergeometric function defined on $x\in \left( -1,1\right) $%
, here $-c\notin \mathbb{N\cup }\left\{ 0\right\} $ and $\left( a\right)
_{n}=\Gamma \left( n+a\right) /\Gamma \left( a\right) $ in which $\Gamma
(x)=\int_{0}^{\infty }t^{x-1}e^{-t}dt$ $(x>0)$ is the gamma function. The
complete elliptic integral of the first kind $\mathcal{K}\left( r\right) $
has the asymptotic formula%
\begin{equation}
\mathcal{K}\left( r\right) \thicksim \ln \frac{4}{r^{\prime }}\text{ \ as }%
r\rightarrow 1^{-},  \label{Kaf}
\end{equation}%
where and in what follows $r^{\prime }=\sqrt{1-r^{2}}$ (see \cite%
{Byrd-HEIES-1971}).

There are many known properties of $\mathcal{K}\left( r\right) $, including
monotonicity, convexity, absolute monotonicity and inequalities. Here we
mention several properties related to the asymptotic formula (\ref{Kaf}),
which are more interesting. Let%
\begin{equation}
\mathcal{F}_{p}\left( x\right) =\left( 1-x\right) ^{p}e^{\mathcal{K}\left( 
\sqrt{x}\right) }\text{ \ and \ }\mathcal{G}_{p}\left( x\right) =e^{\mathcal{%
K}\left( \sqrt{x}\right) }-\frac{p}{\sqrt{1-x}}.  \label{Fp,Gp}
\end{equation}%
In 1996, Qiu and Vamanamurthy \cite[Theorem 1.2]{Qiu-SIAM-JMA-27-1996}
showed that the function $r\mapsto \mathcal{F}_{1/2}\left( r^{2}\right) $ is
strictly decreasing and concave from $\left( 0,1\right) $ onto $\left(
4,e^{\pi /2}\right) $. Yang and Chu \cite[Theorem 3.1]{Yang-MIA-21-2018}
proved that the function $r\mapsto \mathcal{F}_{p/2}\left( r^{2}\right) $ is
strictly increasing on $\left( 0,1\right) $ if and only if $p\leq \pi /4$
and strictly decreasing on $\left( 0,1\right) $ if and only if $p\geq 1$
using a different way. Further, Yang and Tian \cite{Yang-JMI-15-2021} proved
that $-\mathcal{F}_{p}^{\prime }$ is absolutely monotonic on $\left(
0,1\right) $ if and only if $1/2\leq p\leq \left( \pi +4+\sqrt{16-\pi }%
\right) /8=1.340...$, and $\mathcal{F}_{p}$ is absolutely monotonic on $%
\left( 0,1\right) $ if and only if $p\leq \pi /8=0.392...$. Yang and Chu 
\cite[Theorem 3.3]{Yang-MIA-21-2018} proved that $\mathcal{G}_{p}$ is
strictly decreasing on $(0,1)$ if and only if $p\geq 4$ and strictly
increasing on $(0,1)$ if and only if $p\leq \pi e^{\pi /2}/4$.

For the function $f_{p}\left( x\right) =\mathcal{K}(\sqrt{x})/\ln \left( p/%
\sqrt{1-x}\right) $ ($p\geq 1$), Anderson, Vamanamurthy and Vuorinen \cite[
Theorem 2.2]{Anderson-SIAM-JMA-21-1990}\ proved that the function $r\mapsto
f_{c}\left( r^{2}\right) $ is strictly decreasing if and only if $1\leq
c\leq 4$ and increasing if and only if $c\geq e^{2}$. Qiu and Vamanamurthy 
\cite{Qiu-PMJ-5-1995} further verified that $r\mapsto f_{4}\left(
r^{2}\right) $ is strictly concave on $\left( 0,1\right) $. Yand and Tian 
\cite{Yang-AADM-13-2019} showed that $f_{p}$ is concave on $\left(
0,1\right) $ if and only if $p=e^{4/3}$. After three years, Tian and Yang 
\cite[Theorem 1]{Tian-RM-77-2022} proved further that $p=e^{4/3}$ is also
the necessary and sufficient condition for the function $-f_{p}^{\prime
\prime }$ to be absolutely monotonic on $\left( 0,1\right) $ by using
recurrence technique. Recently, Alzer and Richards \cite{Alzer-ITSF-31-2020}
found that the function $x\mapsto 1/f_{p}\left( x\right) $ is strictly
concave on $\left( 0,1\right) $ if and only if $p\geq e^{8/5}$. This
condition was proved in \cite[Theorem 3]{Tian-RM-77-2022} to be also
necessary and sufficient such that the function $x\mapsto -\left[
1/f_{p}\left( x\right) \right] ^{\prime \prime }$ is absolutely monotonic on 
$\left( 0,1\right) $. Moreover, Tian and Yang \cite[Theorem 1]%
{Tian-RM-78-2023} discovered the function $x\mapsto 1/f_{p}\left( x\right) $
is strictly convex on $\left( 0,1\right) $ if and only if $0<p\leq 4$. More
information about monotonicity, convexity and absolute monotonicity for the
functions $f_{p}\left( x\right) $ and its reciprocal can be found in \cite%
{Wang-AADM-14-2020,Yang-AMS-42-2022,Tian-RM-77-2022,Tian-RM-78-2023}, also
refer to \cite{Qiu-SIAM-JMA-30-1999, Yang-MIA-20-2017, Wang-AMS-39B-2019,
Yang-RJ-48-2019,Richards-ITSF-32-2020,Zhao-RACSAM-115-2021,Zhao-JMI-15-2021,Chen-RACSAM-116-2022, Wang-IJPAM-54-2023}
for their generalizations. Besides, several remarkable inequalities related
to the asymptotic formula (\ref{Kaf}) can be seen in \cite%
{Carlson-SIAM-JMA-16-1985,Anderson-SIAM-JMA-23-1992,Qiu-SIAM-JMA-29-1998,Alzer-MPCPS-124(2)-1998,Alzer-JCAM-172-2004, Yang-JIA-2014-192,Yang-MIA-23-2020,Yang-RACSAM-115-2021}%
.

Let us return to (\ref{Fp,Gp}). The absolute monotonicity of the functions $%
\mathcal{F}_{p}$ in \cite{Yang-JMI-15-2021} and the monotonicity of $%
\mathcal{G}_{p}$ in \cite{Yang-MIA-21-2018} encourage us to investigate the
absolute monotonicity of the function $\mathcal{G}_{p}$ on $\left(
0,1\right) $. This seems to be not difficult, however, it is actually quite
challenging. Fortunately, we succeed in solving the problem using the
following new idea:

\textbf{Step 1}: Let%
\begin{equation}
e^{\mathcal{K}\left( \sqrt{x}\right) }=\sum_{n=0}^{n}b_{n}x^{n}\text{ \ and
\ }\frac{1}{\sqrt{1-x}}=\sum_{n=0}^{n}W_{n}x^{n},  \label{ekrx,1/r1-x-ps}
\end{equation}%
where%
\begin{equation}
W_{n}=\frac{\left( 2n-1\right) !!}{\left( 2n\right) !!}=\frac{\Gamma \left(
n+1/2\right) }{\Gamma \left( 1/2\right) \Gamma \left( n+1\right) }
\label{Wn}
\end{equation}%
which satisfies the recurrence relation%
\begin{equation}
W_{n+1}=\frac{n+1/2}{n+1}W_{n}.  \label{Wn-rr}
\end{equation}%
Then%
\begin{equation}
G\left( x\right) =\left[ \frac{\pi }{8}F\left( \frac{1}{2},\frac{1}{2}%
;2;x\right) -\frac{1}{2}\right] e^{\mathcal{K}\left( \sqrt{x}\right)
}=\sum_{n=0}^{\infty }\left[ \left( n+1\right) b_{n+1}-\left( n+\frac{1}{2}%
\right) b_{n}\right] x^{n}.  \label{G}
\end{equation}%
Differentiation yields%
\begin{equation*}
G^{\prime }\left( x\right) =\frac{\pi }{64}e^{\mathcal{K}\left( \sqrt{x}%
\right) }g\left( x\right) ,
\end{equation*}%
where%
\begin{equation}
g\left( x\right) =F\left( \frac{3}{2},\frac{3}{2};3;x\right) +\pi F\left( 
\frac{3}{2},\frac{3}{2};2;x\right) F\left( \frac{1}{2},\frac{1}{2}%
;2;x\right) -4F\left( \frac{3}{2},\frac{3}{2};2;x\right) .  \label{g}
\end{equation}

\textbf{Step 2}: We prove that $g\left( x\right) $ is absolutely monotonic
on $\left( 0,1\right) $. Since $e^{\mathcal{K}\left( \sqrt{x}\right) }$ is
obviously absolutely monotonic on $\left( 0,1\right) $, so is $G^{\prime
}\left( x\right) $.

\textbf{Step 3}: Due to the absolute monotonicity of $G^{\prime }\left(
x\right) $, we deduce the positivity of those coefficients of power series
of $G\left( x\right) $ (\ref{G}), except for the constant term, that is, 
\begin{equation}
\left( n+1\right) b_{n+1}-\left( n+\frac{1}{2}\right) b_{n}>0\text{ \ \ or
equivalently, \ \ }\frac{b_{n+1}}{W_{n+1}}>\frac{b_{n}}{W_{n}}
\label{bn+1/bn>}
\end{equation}%
for $n\geq 1$. Moreover, we also need to prove that $\lim_{n\rightarrow
\infty }\left( b_{n}/W_{n}\right) =4$.

\textbf{Step 4}: From Step 3, we easily arrive at the absolute monotonicity
results for the functions $\pm \mathcal{G}_{p}^{\left( k\right) }\left(
x\right) $ on $\left( 0,1\right) $ for $k\geq 0$, which are contained in the
following theorem.

\begin{theorem}
\label{T}Let $\mathcal{G}_{p}$ be defined on $\left( 0,1\right) $ in (\ref%
{Fp,Gp}). Then the following statements are valid:

(i) $\mathcal{G}_{p}$ is absolutely monotonic on $\left( 0,1\right) $ if and
only if $p\leq \pi e^{\pi /2}/4=3.778...$, and $-\mathcal{G}_{p}$ is
absolutely monotonic on $\left( 0,1\right) $ if and only if $p\geq e^{\pi
/2} $.

(ii) For $k\in \mathbb{N}$, $\mathcal{G}_{p}^{\left( k\right) }$ is
absolutely monotonic on $\left( 0,1\right) $ if and only if $p\leq
b_{k}/W_{k}$, and $-\mathcal{G}_{p}^{\left( k\right) }$ is absolutely
monotonic on $\left( 0,1\right) $ if and only if $p\geq 4$.
\end{theorem}

\begin{remark}
\label{R-AM-iff}The conclusion in Step 3 follows from the following obvious
fact:

A real power series $S\left( x\right) =\sum_{n=0}^{\infty }a_{n}x^{n}$
converging on $\left( -R,R\right) $ ($R>0$) is absolutely monotonic on $%
\left( 0,R\right) $ if and only if $a_{k}\geq 0$ for all $k\geq 0$.

In fact, if $S\left( x\right) $ is absolutely monotonic on $\left(
0,R\right) $, then $S^{\left( k\right) }\left( 0\right) \geq 0$ for all $%
k\geq 0$, which means that $a_{k}\geq 0$ for all $k\geq 0$. Conversely, if $%
a_{k}\geq 0$ for all $k\geq 0$, then $S^{\left( k\right) }\left( x\right)
\geq 0$ for $x\in \left( 0,R\right) $ and all $k\geq 0$.
\end{remark}

\section{Preliminaries and proof of Theorem \protect\ref{T}}

\subsection{Basic knowledge}

In order to prove our results, we need several basic properties of the
hypergeometric function for real $a,b,c$, where $-c\notin \mathbb{N\cup }%
\left\{ 0\right\} $.

(i) Differentiation formula%
\begin{equation}
F^{\prime }\left( a,b;c;x\right) =\frac{ab}{c}F\left( a+1,b+1;c+1;x\right) .
\label{df}
\end{equation}

(ii) The value of $x=1$:%
\begin{equation*}
F\left( a,b;c;1\right) =\dfrac{\Gamma \left( c\right) \Gamma \left(
c-a-b\right) }{\Gamma \left( c-a\right) \Gamma \left( c-b\right) },
\end{equation*}%
where $-c\notin \mathbb{N\cup }\left\{ 0\right\} $, and $c>a+b$. In
particular, we have%
\begin{equation}
F\left( \frac{1}{2},\frac{1}{2};2;1\right) =\frac{4}{\pi }.  \label{F2=}
\end{equation}

(iii) Linear transformation formula%
\begin{equation}
F\left( a,b;c;x\right) =(1-x)^{c-a-b}F(c-a,c-b;c;x).  \label{Ltf}
\end{equation}

(iv) Asymptotic formula%
\begin{equation}
F\left( a,b;a+b;x\right) =\dfrac{R\left( a,b\right) -\ln \left( 1-x\right) }{%
B\left( a,b\right) }+O\left( \left( 1-x\right) \ln \left( 1-x\right) \right) 
\text{ as }x\rightarrow 1^{-},  \label{F-near1}
\end{equation}%
where%
\begin{equation*}
B\left( a,b\right) =\frac{\Gamma \left( a\right) \Gamma \left( b\right) }{%
\Gamma \left( a+v\right) }\text{, \ }a,b>0
\end{equation*}%
is the classical beta function,%
\begin{equation}
R\equiv R\left( a,b\right) =-2\gamma -\psi \left( a\right) -\psi \left(
b\right) \text{,}  \label{R(a,b)}
\end{equation}%
here $\psi \left( x\right) =\Gamma ^{\prime }\left( x\right) /\Gamma \left(
x\right) $, $x>0$ is the psi function and $\gamma $ is the Euler-Mascheroni
constant. In particular, we have%
\begin{equation}
F\left( \frac{3}{2},\frac{3}{2};3;x;\right) \thicksim \frac{8}{\pi }\left[
4\ln 2-4-\ln \left( 1-x\right) \right] \rightarrow \infty \text{ as }%
x\rightarrow 1^{-}.  \label{F3/2,3/2,3-00}
\end{equation}

Moreover, using the notation $W_{n}$, certain specific hypergeometric
functions can be simply represented as%
\begin{equation}
\begin{array}{ll}
F\left( \dfrac{1}{2},\dfrac{1}{2};1;x\right) =\dsum\limits_{n=0}^{\infty
}W_{n}^{2}x^{n}, & F\left( \dfrac{3}{2},\dfrac{3}{2};2;x\right)
=4\dsum\limits_{n=0}^{\infty }\left( n+1\right) W_{n+1}^{2}x^{n},\bigskip \\ 
F\left( \dfrac{1}{2},\dfrac{1}{2};2;x\right) =\dsum\limits_{n=0}^{\infty }%
\dfrac{W_{n}^{2}}{n+1}x^{n}, & F\left( \dfrac{3}{2},\dfrac{3}{2};3;x\right)
=8\dsum\limits_{n=0}^{\infty }\dfrac{n+1}{n+2}W_{n+1}^{2}x^{n},%
\end{array}
\label{sF-Wn}
\end{equation}

\subsection{The first three steps}

In this subsection, we complete the proofs of Step 1--3 introduced in
Section 1.

\begin{lemma}[Step 1]
\label{L-eK-ps}Let $b_{n}$ be defined as in (\ref{ekrx,1/r1-x-ps}) Then the
formula (\ref{G}) is valid. Moreover, $b_{n}$ satisfies the recurrence
formula: $b_{0}=e^{\pi /2}$ and for $n\geq 0$,%
\begin{equation}
b_{n+1}=\frac{n}{n+1}b_{n}+\frac{\pi }{8\left( n+1\right) }\sum_{k=0}^{n}%
\frac{W_{k}^{2}}{k+1}b_{n-k}.  \label{bn+1-rr}
\end{equation}%
In particular, the first few coefficients are:%
\begin{equation}
b_{1}=\frac{1}{8}\pi e^{\pi /2}\text{, \ }b_{2}=\frac{1}{128}\pi \left( \pi
+9\right) e^{\pi /2}\text{, \ and \ }b_{3}=\frac{\pi }{3072}\left( \pi
^{2}+27\pi +150\right) e^{\pi /2}.  \notag
\end{equation}
\end{lemma}

\begin{proof}
Differentiation yields%
\begin{equation*}
\left[ e^{\mathcal{K}\left( \sqrt{x}\right) }\right] ^{\prime }=\frac{\pi }{8%
}F\left( \frac{3}{2},\frac{3}{2};2;x\right) e^{\mathcal{K}\left( \sqrt{x}%
\right) }=\sum_{n=0}^{\infty }nb_{n}x^{n-1},
\end{equation*}%
which, by (\ref{Ltf}), implies that%
\begin{equation}
\frac{\pi }{8}F\left( \frac{1}{2},\frac{1}{2};2;x\right) e^{\mathcal{K}%
\left( \sqrt{x}\right) }=\left( 1-x\right) \sum_{n=0}^{\infty
}nb_{n}x^{n-1}=\sum_{n=0}^{\infty }\left[ \left( n+1\right) b_{n+1}-nb_{n}%
\right] x^{n}.  \label{FeK-ps}
\end{equation}%
Hence,%
\begin{equation*}
G\left( x\right) =\frac{\pi }{8}F\left( \frac{1}{2},\frac{1}{2};2;x\right)
e^{\mathcal{K}\left( \sqrt{x}\right) }-\frac{1}{2}e^{\mathcal{K}\left( \sqrt{%
x}\right) }=\sum_{n=0}^{\infty }\left[ \left( n+1\right) b_{n+1}-\left( n+%
\frac{1}{2}\right) b_{n}\right] x^{n},
\end{equation*}%
which proves the formula (\ref{G}).

On the other hand, applying the third formula of (\ref{sF-Wn}) and Cauchy
product formula to (\ref{FeK-ps}) gives 
\begin{equation*}
\frac{\pi }{8}\sum_{n=0}^{\infty }\left( \sum_{k=0}^{n}\frac{W_{k}^{2}}{k+1}%
b_{n-k}\right) x^{n}=\sum_{n=0}^{\infty }\left[ \left( n+1\right)
b_{n+1}-nb_{n}\right] x^{n},
\end{equation*}%
then comparing coefficients of $x^{n}$ leads to%
\begin{equation*}
\frac{\pi }{8}\sum_{k=0}^{n}\frac{W_{k}^{2}}{k+1}b_{n-k}=\left( n+1\right)
b_{n+1}-nb_{n},
\end{equation*}%
which implies (\ref{bn+1-rr}). Letting $x\rightarrow 0^{+}$ leads us to $%
b_{0}=e^{\pi /2}$, and the first few coefficients follows from the
recurrence formula (\ref{bn+1-rr}), which completes the proof.
\end{proof}

\begin{lemma}[Step 2]
\label{L-dG-AM}Let $G\left( x\right) $ be defined by (\ref{G}). Then $%
G^{\prime }\left( x\right) $ is absolutely monotonic on $\left( 0,1\right) $.
\end{lemma}

\begin{proof}
Differentiation yields%
\begin{equation*}
G^{\prime }\left( x\right) =\frac{\pi }{64}e^{\mathcal{K}\left( \sqrt{x}%
\right) }g\left( x\right) ,
\end{equation*}%
where $g\left( x\right) $ is defined by (\ref{g}). Since $e^{\mathcal{K}%
\left( \sqrt{x}\right) }$ is obviously absolutely monotonic on $\left(
0,1\right) $, if we prove that $g\left( x\right) $ is absolutely monotonic
on $\left( 0,1\right) $, then by Remark \ref{R-AM-iff}, so is $G^{\prime
}\left( x\right) $.

Let%
\begin{equation*}
g_{0}\left( x\right) =\left( 1-x\right) F\left( \frac{3}{2},\frac{3}{2}%
;3;x\right) +\pi F^{2}\left( \frac{1}{2},\frac{1}{2};2;x\right) -4F\left( 
\frac{1}{2},\frac{1}{2};2;x\right) .
\end{equation*}%
By the third and fourth formulas of (\ref{sF-Wn}), $g_{0}\left( x\right) $
can be expressed as%
\begin{equation*}
g_{0}\left( x\right) =\sum_{n=0}^{\infty }u_{n}x^{n},
\end{equation*}%
where%
\begin{equation}
u_{n}=\pi \sum_{k=0}^{n}\frac{W_{k}^{2}W_{n-k}^{2}}{\left( k+1\right) \left(
n-k+1\right) }-6\frac{2n+1}{\left( n+2\right) \left( n+1\right) }W_{n}^{2}.
\label{un}
\end{equation}%
It was proved in \cite[Lemma 2.5]{Yang-MIA-21-2018} that $u_{n}<0$ for all $%
n\geq 8$. This together with \cite[Eqs. (3.6)--(3.8)]{Yang-MIA-21-2018}
gives that $u_{0}=\pi -3>0$, $u_{1}=\left( \pi -3\right) /4>0$ and $u_{n}<0$
for all $n\geq 2$. Moreover, by\ (\ref{F2=}) and (\ref{F3/2,3/2,3-00}), we
have%
\begin{equation*}
g_{0}\left( 1^{-}\right) =\sum_{n=0}^{\infty }u_{n}=0.
\end{equation*}

On the other hand, applying the linear transformation formula (\ref{Ltf})
and the Cauchy product formula gives%
\begin{eqnarray*}
g\left( x\right) &=&F\left( \frac{3}{2},\frac{3}{2};3;x\right) +\pi F\left( 
\frac{3}{2},\frac{3}{2};2;x\right) F\left( \frac{1}{2},\frac{1}{2}%
;2;x\right) -4F\left( \frac{3}{2},\frac{3}{2};2;x\right) \\
&=&\frac{g_{0}\left( x\right) }{1-x}=\frac{1}{1-x}\sum_{n=0}^{\infty
}u_{n}x^{n}=\sum_{n=0}^{\infty }\left( \sum_{k=0}^{n}u_{k}\right)
x^{n}:=\sum_{n=0}^{\infty }v_{n}x^{n}.
\end{eqnarray*}%
It is verified that $v_{0}=u_{0}=\pi -3>0$, $v_{1}=u_{0}+u_{1}=5\left( \pi
-3\right) /4>0$. For $n\geq 2$, since%
\begin{equation*}
v_{n}-v_{n-1}=\sum_{k=0}^{n}u_{k}-\sum_{k=0}^{n-1}u_{k}=u_{n}<0,
\end{equation*}%
it is readily deduced that%
\begin{equation*}
v_{n}>\lim_{n\rightarrow \infty }v_{n}=\sum_{n=0}^{\infty }u_{n}=0.
\end{equation*}%
This proves the absolute monotonicity of $g\left( x\right) $ on $\left(
0,1\right) $, thereby completing the proof.
\end{proof}

\begin{lemma}[Step 3]
\label{L-bn/Wn-i}Let $b_{n}$ be defined by (\ref{bn+1-rr}). Then the
sequence $\left\{ b_{n}/W_{n}\right\} _{n\geq 1}$ is increasing with%
\begin{equation}
\lim_{n\rightarrow \infty }\frac{b_{n}}{W_{n}}=4.  \label{bn/Wn->4}
\end{equation}
\end{lemma}

\begin{proof}
To prove the increasing property of the sequence $\left\{
b_{n}/W_{n}\right\} _{n\geq 1}$, it suffices to prove the inequality 
\begin{equation*}
b_{n+1}>\frac{W_{n+1}}{W_{n}}b_{n}=\frac{n+1/2}{n+1}b_{n},
\end{equation*}%
or equivalently, 
\begin{equation}
\left( n+1\right) b_{n+1}-\left( n+\frac{1}{2}\right) b_{n}>0
\label{bn/Wn-ip}
\end{equation}%
for $n\in \mathbb{N}$. By Lemma \ref{L-dG-AM}, $G^{\prime }\left( x\right) $
is absolutely monotonic on $\left( 0,1\right) $, which together with (\ref{G}%
) and Remark \ref{R-AM-iff}\ implies (\ref{bn/Wn-ip}).

Since the sequence $\left\{ b_{n}/W_{n}\right\} _{n\geq 1}$ is increasing
and $b_{1}/W_{1}>0$, the $\lim\limits_{n\rightarrow \infty
}(b_{n}/W_{n})=p_{0}>0$ or tends to $\infty $. We first prove that $%
\lim_{n\rightarrow \infty }\left( b_{n}/W_{n}\right) \leq 4$. If not, that
is, $\lim_{n\rightarrow \infty }\left( b_{n}/W_{n}\right) >4$, then there
exists $n_{0}>1$ such that $b_{n}/W_{n}>4$ for $n\geq n_{0}$. Let us denote
by $p^{\ast }=b_{n_{0}}/W_{n_{0}}$. The increasing property of $%
\{b_{n}/W_{n}\}_{n\geq 1}$ makes us to deduce 
\begin{equation*}
\frac{b_{n}}{W_{n}}-p^{\ast }=\frac{b_{n}}{W_{n}}-\frac{b_{n_{0}}}{W_{n_{0}}}%
\geq 0\text{, that is,}\quad b_{n}-p^{\ast }W_{n}\geq 0\text{ for }n\geq
n_{0}\text{,}
\end{equation*}%
which in conjunction with (\ref{ekrx,1/r1-x-ps}) yields 
\begin{equation*}
\mathcal{G}_{p^{\ast }}(x)=\sum_{n=0}^{\infty }\left( b_{n}-p^{\ast
}W_{n}\right) x^{n}\geq \sum_{n=0}^{n_{0}-1}(b_{n}-p^{\ast }W_{n})x^{n},
\end{equation*}%
which contradicts with 
\begin{equation*}
\mathcal{G}_{p^{\ast }}(1^{-})=\lim_{x\rightarrow 1^{-}}\left[ e^{\ln (4/%
\sqrt{1-x})}-\frac{p^{\ast }}{\sqrt{1-x}}\right] =-\infty .
\end{equation*}%
If $\lim\limits_{n\rightarrow \infty }(b_{n}/W_{n})=p_{0}<4$, that is to
say, $\sup_{n\geq 1}\{b_{n}/W_{n}\}=p_{0}$, then $b_{n}/W_{n}\leq p_{0}$ for
all $n\geq 1$. Similarly, 
\begin{equation*}
\mathcal{G}_{p_{0}}(x)=\sum_{n=0}^{\infty }\left( b_{n}-p_{0}W_{n}\right)
x^{n}\leq b_{0}-p_{0}W_{0}
\end{equation*}%
which is a contradiction with $\mathcal{G}_{p_{0}}(1^{-})=\infty $ due to $%
p_{0}<4$. This proves $\lim\limits_{n\rightarrow \infty }(b_{n}/W_{n})=4$,
and the proof is done.
\end{proof}

\subsection{Step 4: Proof of Theorem \protect\ref{T}}

By means of Lemma (Step 3) \ref{L-bn/Wn-i}, we can easily complete the proof
of our main result.

\begin{proof}[\textit{Proof of Theorem \protect\ref{T}}]
From (\ref{ekrx,1/r1-x-ps}), $\mathcal{G}_{p}\left( x\right) $ can be
expressed in power series: 
\begin{equation}
\mathcal{G}_{p}\left( x\right) =\sum_{n=0}^{\infty
}b_{n}x^{n}-p\sum_{n=0}^{\infty }W_{n}x^{n}=\sum_{n=0}^{\infty }\left(
b_{n}-pW_{n}\right) x^{n}:=\sum_{n=0}^{\infty }c_{n}\left( p\right) x^{n}.
\label{Gp-ps}
\end{equation}

As a matter to remember, Lemma \ref{L-bn/Wn-i} states that the sequence $%
\{b_{n}/W_{n}\}$ is strictly increasing for $n\geq 1$ with $%
\lim\limits_{n\rightarrow \infty }(b_{n}/W_{n})=4$.

(i) By (\ref{Gp-ps}), $\mathcal{G}_{p}$ is absolutely monotonic on $\left(
0,1\right) $ if and only if 
\begin{equation*}
c_{n}(p)=b_{n}-pW_{n}\geq 0\text{ for}\quad n\geq 0,
\end{equation*}%
which implies that 
\begin{equation*}
p\leq \inf_{n\geq 0}\left\{ \frac{b_{n}}{W_{n}}\right\} =\min \left\{ \frac{%
b_{0}}{W_{0}},\frac{b_{1}}{W_{1}}\right\} =\min \left\{ e^{\pi /2},\frac{\pi 
}{4}e^{\pi /2}\right\} =\frac{\pi }{4}e^{\pi /2}.
\end{equation*}%
Similarly, $-\mathcal{G}_{p}$ is absolutely monotonic on $\left( 0,1\right) $
if and only if 
\begin{equation*}
p\geq \sup_{n\geq 0}\left\{ \frac{b_{n}}{W_{n}}\right\} =\max \left\{ \frac{%
b_{0}}{W_{0}},4\right\} =\max \left\{ e^{\pi /2},4\right\} =e^{\pi /2}.
\end{equation*}%
(ii) In a fashion to \textit{(i)}, $\mathcal{G}_{p}^{(k)}$ for $k\geq 1$ is
absolutely monotonic on $(0,1)$ if and only if 
\begin{equation*}
p\leq \inf_{n\geq k}\left\{ \frac{b_{n}}{W_{n}}\right\} =\frac{b_{k}}{W_{k}};
\end{equation*}%
and $-\mathcal{G}_{p}^{(k)}$ is absolutely monotonic on $(0,1)$ if and only
if 
\begin{equation*}
p\geq \sup_{n\geq k}\left\{ \frac{b_{n}}{W_{n}}\right\} =4.
\end{equation*}

This completes the proof.
\end{proof}

\section{New functional inequalities related to $\mathcal{K}(r)$}

In this Section, Theorem \ref{T} will be applied to establish several
functional inequalities involving the complete elliptic integral of the
first kind $\mathcal{K}(r)$.

\begin{proposition}
\label{P1} Let $b_{n}$ be defined in (\ref{ekrx,1/r1-x-ps}) with $%
b_{0}=e^{\pi /2}$ and $c_{n}(p)=b_{n}-pW_{n}$ for $n\geq 0$. The double
inequality 
\begin{equation}
\ln \left( \frac{p}{r^{\prime }}+\sum_{n=0}^{m}c_{n}(p)r^{2n}\right) <%
\mathcal{K}(r)<\ln \left( \frac{q}{r^{\prime }}+\sum_{n=0}^{m}c_{n}(q)r^{2n}%
\right)  \label{K<>I1}
\end{equation}%
holds for $r\in (0,1)$ if and only if $p\leq b_{m+1}/W_{m+1}$ and $q\geq 4$.
\end{proposition}

\begin{proof}
As shown in (\ref{Gp-ps}), we denote by 
\begin{equation}
\mathcal{R}_{p,m}(x):=e^{\mathcal{K}(\sqrt{x})}-\frac{p}{\sqrt{1-x}}%
-\sum_{n=0}^{m}c_{n}(p)x^{n}=\sum_{n=m+1}^{\infty }c_{n}(p)x^{n}.
\label{Rp,m}
\end{equation}%
Theorem \ref{T} (ii) enables us to know that $c_{n}(p)\geq 0$ for $n\geq m+1$
if $p\leq b_{m+1}/W_{m+1}$ and $c_{n}(q)\leq 0$ for $n\geq m+1$ if $q\geq 4$%
. This together with (\ref{Rp,m}) gives the desired inequality (\ref{K<>I1}).

Conversely, the necessary conditions that (\ref{K<>I1}) holds for all $r\in
(0,1)$ require to satisfy $\mathcal{R}_{p,m}(0^{+})\geq 0$ and $\mathcal{R}%
_{q,m}(1^{-})\leq 0$, which imply $c_{m+1}(p)=b_{m+1}-pW_{m+1}\geq 0$, (i.e. 
$p\leq b_{m+1}/W_{m+1}$) and 
\begin{align*}
\mathcal{R}_{q,m}(1^{-})& =\lim_{x\rightarrow 1^{-}}\frac{\sqrt{1-x}e^{%
\mathcal{K}(\sqrt{x})}-q-\sqrt{1-x}\sum\limits_{n=0}^{m}c_{n}(p)x^{n}}{\sqrt{%
1-x}} \\
& =\lim_{x\rightarrow 1^{-}}\frac{4-q}{\sqrt{1-x}}\leq 0,
\end{align*}%
namely, $q\geq 4$.
\end{proof}

\begin{remark}
Taking $m=0$ in the double inequality (\ref{K<>I1}) leads to 
\begin{equation}
\ln \left( e^{\pi /2}-p+\frac{p}{r^{\prime }}\right) <\mathcal{K}(r)<\ln
\left( e^{\pi /2}-q+\frac{q}{r^{\prime }}\right)  \label{K<>I1-m=0}
\end{equation}%
for all $r\in (0,1)$ with the best constants $p=b_{1}/W_{1}=\pi e^{\pi /2}/4$
and $q=4$. This result was proven in \cite[Corollary 3.7]{Yang-MIA-21-2018}.
Clearly, Proposition \ref{P1} is a generalization of \cite[Corollary 3.7]%
{Yang-MIA-21-2018}.
\end{remark}

\begin{remark}
Taking $m=1$ in (\ref{K<>I1}) yields 
\begin{equation}
\ln \left[ e^{\pi /2}-p-\left( \frac{p}{2}-\frac{\pi }{8}e^{\pi /2}\right)
r^{2}+\frac{p}{r^{\prime }}\right] <\mathcal{K}(r)<\ln \left[ e^{\pi
/2}-q-\left( \frac{q}{2}-\frac{\pi }{8}e^{\pi /2}\right) r^{2}+\frac{q}{%
r^{\prime }}\right]  \label{K<>I1-m=1}
\end{equation}%
for all $r\in (0,1)$ with the best constants $p=b_{2}/W_{2}=\left[ \pi (\pi
+9)e^{\pi /2}\right] /48$ and $q=4$. Since 
\begin{align*}
& \left[ e^{\pi /2}-\frac{b_{2}}{W_{2}}-\left( \frac{b_{2}}{2W_{2}}-\frac{%
\pi }{8}e^{\pi /2}\right) r^{2}+\frac{b_{2}}{W_{2}r^{\prime }}\right]
-\left( e^{\pi /2}-\frac{b_{1}}{W_{1}}+\frac{b_{1}}{W_{1}r^{\prime }}\right)
\\
& =\frac{\pi (\pi -3)e^{\pi /2}}{96}\frac{r^{4}(r^{2}+3)}{r^{\prime }\left[
2+(r^{2}+2)r^{\prime }\right] }>0
\end{align*}%
and $2-\pi e^{\pi /2}/8=0.110\cdots >0$, we conclude that the optimal
inequality of (\ref{K<>I1-m=1}) improves the one of (\ref{K<>I1-m=0}).
\end{remark}

For $n\geq 0$, Theorem \ref{T} (ii) demonstrates that the function $x\mapsto 
\mathcal{R}_{4,n}(x)/x^{n+1}$ is strictly decreasing on $(0,1)$ together
with 
\begin{equation*}
\lim_{x\rightarrow 0^{+}}\frac{\mathcal{R}_{4,n}(x)}{x^{n+1}}%
=c_{n+1}(4)\quad \text{and}\quad \lim_{x\rightarrow 1^{-}}\frac{\mathcal{R}%
_{4,n}(x)}{x^{n+1}}=-\sum_{k=0}^{n}c_{k}(4).
\end{equation*}%
This yields 
\begin{equation*}
-\left( \sum_{n=0}^{m}c_{n}(4)\right) x^{m+1}<e^{\mathcal{K}(\sqrt{x})}-%
\frac{4}{\sqrt{1-x}}-\sum_{n=0}^{m}c_{n}(4)x^{n}<c_{m+1}(4)x^{m+1}
\end{equation*}%
for $x\in (0,1)$. By replacing $x=r^{2}$, we obtain therefore the following
proposition.

\begin{proposition}
\label{P2} Let $b_{n}$ be defined by (\ref{ekrx,1/r1-x-ps}) and $%
c_{n}=b_{n}-4W_{n}$ for $n\geq 0$. The double inequality 
\begin{equation}
\ln \left[ \frac{4}{r^{\prime }}+\sum_{n=0}^{m}c_{n}r^{2n}-\left(
\sum_{n=0}^{m}c_{n}\right) r^{2m+2}\right] <\mathcal{K}(r)<\ln \left[ \frac{4%
}{r^{\prime }}+\sum_{n=0}^{m+1}c_{n}r^{2n}\right]  \label{K<>I2}
\end{equation}%
holds for $r\in (0,1)$, where the lower and upper bounds are sharp.
\end{proposition}

\begin{remark}
\label{rmk4} Taking $m=0$ in Proposition \ref{P2} yields 
\begin{equation}
\ln \left[ \frac{4}{r^{\prime }}+\left( e^{\pi /2}-4\right) r^{\prime 2}%
\right] <\mathcal{K}(r)<\ln \left[ \frac{4}{r^{\prime }}+\left( e^{\pi
/2}-4\right) -\left( 2-\frac{\pi }{8}e^{\pi /2}\right) r^{2}\right]
\label{K<>I2-m=0}
\end{equation}%
for $r\in (0,1)$. It was proved in \cite[Corollary 3.5]{Qiu-SIAM-JMA-29-1998}%
\ that%
\begin{equation}
\ln \left[ \frac{4}{r^{\prime }}+\left( e^{\pi /2}-4\right) r^{\prime }%
\right] <\mathcal{K}(r)<\ln \left( \frac{4}{r^{\prime }}+e^{\pi /2}-4\right)
\label{K-QI}
\end{equation}%
for $r\in \left( 0,1\right) $. In \cite[Eq. (3.5)]{Yang-MIA-21-2018}, it was
found that 
\begin{equation}
\ln \left[ \frac{4}{r^{\prime }}+\left( e^{\pi /2}-4\right) r^{\prime }%
\right] <\mathcal{K}(r)<\ln \left[ \frac{(\pi +4)e^{\pi /2}}{8r^{\prime }}%
+\left( \frac{1}{2}-\frac{\pi }{8}\right) e^{\pi /2}r^{\prime }\right]
\label{K-YI}
\end{equation}%
for $r\in (0,1)$. The left side inequality of (\ref{K<>I2-m=0}) looks like
the one of (\ref{K-QI}) but a bit worse than the left side of (\ref%
{K<>I2-m=0}), while the right side of (\ref{K<>I2-m=0}) is better than the
ones of (\ref{K-QI}) and of (\ref{K-YI}). This can be confirmed by%
\begin{equation*}
\left[ \frac{4}{r^{\prime }}+\left( e^{\pi /2}-4\right) -\left( 2-\frac{\pi 
}{8}e^{\pi /2}\right) r^{2}\right] -\left[ \frac{4}{r^{\prime }}+\left(
e^{\pi /2}-4\right) \right] <0
\end{equation*}%
and%
\begin{align*}
& \left[ \frac{4}{r^{\prime }}+e^{\pi /2}-4-\left( 2-\frac{\pi }{8}e^{\pi
/2}\right) r^{2}\right] -\left[ \frac{(\pi +4)e^{\pi /2}}{8r^{\prime }}%
+\left( \frac{1}{2}-\frac{\pi }{8}\right) e^{\pi /2}r^{\prime }\right] \\
& =-\frac{\left( 1-r^{\prime }\right) ^{2}}{8r^{\prime }}\left[ \left(
\left( \pi +4\right) e^{\pi /2}-32\right) -\left( 16-\pi e^{\pi /2}\right)
r^{\prime }\right] \\
& <-\frac{\left( \pi +2\right) e^{\pi /2}-24}{4}\frac{\left( 1-r^{\prime
}\right) ^{2}}{r^{\prime }}<0
\end{align*}%
for $r\in (0,1)$.
\end{remark}

It was shown in \cite{Petrovic-PMUB-1-1932} that every convex function $%
f:[0,\infty )\mapsto \mathbb{R}$ satisfies a functional inequality 
\begin{equation}
f(x)+f(y)\leq f(0)+f(x+y)  \label{convex}
\end{equation}%
for $x,y\in \lbrack 0,\infty )$. In particular, $f(x)$ is superadditive if $%
f(0)\leq 0$. As a matter to remind, a function $f(x)$ defined on $\mathbb{R}%
_{+}\backslash (0,a]$ is said to be superadditive if 
\begin{equation}
f(x)+f(y)\leq f(x+y)  \label{super}
\end{equation}%
for $x,y\in (a,\infty )$, while it is called subadditive if the inequality (%
\ref{super}) reverses.

Let $\mathcal{R}_{p,m}$ be defined as in (\ref{Rp,m}). Then it is clear that 
$\mathcal{R}_{p,m}(0^{+})=0$ for $m\geq 0$. Since 
\begin{equation*}
\mathcal{R}_{p,m}^{\prime \prime }(x)=\left[ \sum_{n=m+1}^{\infty
}c_{n}(p)x^{n}\right] ^{\prime \prime }=\left\{ 
\begin{array}{cc}
\sum\limits_{n=2}^{\infty }n\left( n-1\right) c_{n}\left( p\right) x^{n-2} & 
\text{if }m=0, \\ 
\sum\limits_{n=m+1}^{\infty }n\left( n-1\right) c_{n}\left( p\right) x^{n-2}
& \text{if }m\geq 1,%
\end{array}%
\right.
\end{equation*}%
it follows from Theorem \ref{T} (ii) that $c_{n}(p)\geq 0$ for $n\geq 2$ if $%
p\leq b_{2}/W_{2}=[\pi (\pi +9)e^{\pi /2}]/48$, while $c_{n}(p)\leq 0$ for $%
n\geq 2$ if $p\geq 4$. This together with (\ref{convex}) deduces that the
double inequality 
\begin{equation*}
2\mathcal{R}_{p,m}\left( \frac{x+y}{2}\right) <\mathcal{R}_{p,m}(x)+\mathcal{%
R}_{p,m}(y)<\mathcal{R}_{p,m}(x+y)
\end{equation*}%
holds for $x,y,x+y\in (0,1)$ if $p\leq \lbrack \pi (\pi +9)e^{\pi /2}]/48$
and it reverses if $p\geq 4$.

\begin{proposition}
\label{P3} Let $b_{n}$ be defined in (\ref{ekrx,1/r1-x-ps}) and $%
c_{n}(p)=b_{n}-pW_{n}$. If $p\leq \lbrack \pi (\pi +9)e^{\pi /2}]/48$, then
the double inequality 
\begin{align*}
& 2e^{\mathcal{K}\left( \sqrt{\frac{x+y}{2}}\right) }+p\left( \frac{1}{\sqrt{%
1-x}}+\frac{1}{\sqrt{1-y}}-\frac{2}{\sqrt{1-\frac{x+y}{2}}}\right)
+\sum_{n=0}^{m}c_{n}(p)\left[ x^{n}+y^{n}-2\left( \frac{x+y}{2}\right) ^{n}%
\right] \\
& <e^{\mathcal{K}\left( \sqrt{x}\right) }+e^{\mathcal{K}\left( \sqrt{y}%
\right) } \\
& <e^{\mathcal{K}\left( \sqrt{x+y}\right) }+p\left( \frac{1}{\sqrt{1-x}}+%
\frac{1}{\sqrt{1-y}}-\frac{1}{\sqrt{1-x-y}}\right) +\sum_{n=0}^{m}c_{n}(p)%
\left[ x^{n}+y^{n}-\left( x+y\right) ^{n}\right]
\end{align*}%
holds for $x,y,x+y\in (0,1)$, and it reverses if $p\geq 4$.
\end{proposition}

Taking $m=0$ into Proposition \ref{P3}, we obtain

\begin{corollary}
\label{C-P3-m=0} If $p\leq \lbrack \pi (\pi +9)e^{\pi /2}]/48$, then the
double inequality 
\begin{align*}
& 2e^{\mathcal{K}\left( \sqrt{\frac{x+y}{2}}\right) }+p\left( \frac{1}{\sqrt{%
1-x}}+\frac{1}{\sqrt{1-y}}-\frac{2}{\sqrt{1-\frac{x+y}{2}}}\right) <e^{%
\mathcal{K}\left( \sqrt{x}\right) }+e^{\mathcal{K}\left( \sqrt{y}\right) } \\
& \hspace{3cm}<e^{\mathcal{K}\left( \sqrt{x+y}\right) }+p\left( \frac{1}{%
\sqrt{1-x}}+\frac{1}{\sqrt{1-y}}-\frac{1}{\sqrt{1-x-y}}-1\right) +e^{\pi /2}
\end{align*}%
holds for $x,y,x+y\in (0,1)$, and it is reversed if $p\geq 4$.
\end{corollary}

\begin{remark}
Let $p=4$, $y=1-r^{2}$ and $x\rightarrow r^{2}$ in Corollary \ref{C-P3-m=0}.
Then we obtain 
\begin{equation*}
e^{\pi /2}-4<e^{\mathcal{K}(r)}+e^{\mathcal{K}(r^{\prime })}-4\left( \frac{1%
}{r}+\frac{1}{r^{\prime }}\right) \leq 2e^{\mathcal{K}(1/\sqrt{2})}-8\sqrt{2}
\end{equation*}%
for $r\in (0,1)$, where 
\begin{equation*}
\mathcal{K}\left( \frac{1}{\sqrt{2}}\right) =\frac{\pi }{2}F\left( \frac{1}{2%
},\frac{1}{2};1;\frac{1}{2}\right) =\frac{\Gamma (1/4)^{2}}{4\sqrt{\pi }},
\end{equation*}%
which can be computed by the formula (see \cite[Eq. (15.1.26)]%
{Abramowitz-HMFFGMT-1972})%
\begin{equation}
F\left( a,1-a;b;\frac{1}{2}\right) =\frac{2^{1-b}\sqrt{\pi }\Gamma \left(
b\right) }{\Gamma \left( a/2+b/2\right) \Gamma \left( 1/2+b/2-a/2\right) }.
\label{F-x=1/2}
\end{equation}%
Clearly, the lower and upper bounds are sharp by (\ref{Kaf}).
\end{remark}

The following lemma was proved in \cite[Lemma 4]{Tian-RM-77-2022}, where $h$
should be infinitely differentiable instead of \textquotedblleft thrice
differentiable\textquotedblright .

\begin{lemma}
\label{L-p-cm}Let the function $h$ be defined on $\mathbb{R}$ which is
infinitely differentiable. If $h^{\left( 2k+1\right) }\left( x\right)
>\left( <\right) 0$ for $x\in \mathbb{R}$ and $k\in \mathbb{N}$ then the
function%
\begin{equation*}
\phi \left( x\right) =\frac{h\left( 2c-x\right) -h\left( x\right) }{2c-2x}%
\text{ if }x\neq c\text{ \ and \ }\phi \left( c\right) =h^{\prime }\left(
c\right)
\end{equation*}%
satisfies that $\phi ^{\left( 2k-2\right) }\left( x\right) $ is strictly
convex (concave) on $\mathbb{R}$, and is decreasing (increasing) on $\left(
-\infty ,c\right) $ and increasing (decreasing) on $\left( c,\infty \right) $%
. Furthermore, if $h^{\prime }\left( x\right) $ is absolutely monotonic on $%
\mathbb{R}$ then $\phi \left( x\right) $ is completely monotonic on $\left(
-\infty ,c\right) $ and absolutely monotonic on $\left( c,\infty \right) $.
\end{lemma}

Applying Lemma \ref{L-p-cm} and the absolute monotonicity of $-\mathcal{G}%
_{4}^{\prime }\left( x\right) $, we have the following statement.

\begin{proposition}
The function%
\begin{equation*}
\mathcal{H}\left( x\right) =\frac{\mathcal{G}_{4}\left( x\right) -\mathcal{G}%
_{4}\left( 1-x\right) }{1-2x}=\frac{1}{1-2x}\left[ e^{\mathcal{K}\left( 
\sqrt{x}\right) }-e^{\mathcal{K}\left( \sqrt{1-x}\right) }+\frac{4}{\sqrt{x}}%
-\frac{4}{\sqrt{1-x}}\right]
\end{equation*}%
is completely monotonic on $\left( 0,1/2\right) $ and absolutely monotonic
on $\left( 1/2,1\right) $. Consequently, the double inequality%
\begin{equation}
\beta \left( 1-2r^{2}\right) <e^{\mathcal{K}\left( r\right) }-e^{\mathcal{K}%
\left( r^{\prime }\right) }+\frac{4}{r}-\frac{4}{r^{\prime }}<\alpha \left(
1-2r^{2}\right)  \label{eK-eK'<>}
\end{equation}%
holds for $r\in \left( 0,1/\sqrt{2}\right) $ with the best constants $\alpha
=e^{\pi /2}-4\approx 0.810$ and%
\begin{equation*}
\beta =4\sqrt{2}-\frac{1}{\sqrt{\pi }}\Gamma ^{2}\left( \frac{3}{4}\right)
\exp \left( \frac{\Gamma \left( 1/4\right) ^{2}}{4\sqrt{\pi }}\right)
\approx 0.246.
\end{equation*}%
The double inequality (\ref{eK-eK'<>}) reverses for $r\in \left( 1/\sqrt{2}%
,1\right) $.
\end{proposition}

\begin{proof}
Taking $h\left( x\right) =-\mathcal{G}_{4}\left( x\right) $ and $c=1/2$ in
Lemma \ref{L-p-cm} yields the complete and absolute monotonicity of $%
\mathcal{H}\left( x\right) $ on $\left( 0,1/2\right) $ and $\left(
1/2,1\right) $, respectively. The decreasing property of $\mathcal{H}\left(
x\right) $ on $\left( 0,1/2\right) $ implies that%
\begin{equation}
-\mathcal{G}_{4}^{\prime }\left( \frac{1}{2}\right) =\mathcal{H}\left( \frac{%
1}{2}\right) <\mathcal{H}\left( x\right) =\mathcal{H}\left( 0\right) =e^{\pi
/2}-4  \label{H<>}
\end{equation}%
for $x\in \left( 0,1/2\right) $. Using the formula (\ref{F-x=1/2}) we have%
\begin{eqnarray*}
F\left( \frac{1}{2},\frac{1}{2};1;\frac{1}{2}\right) &=&\frac{1}{2\pi ^{3/2}}%
\Gamma ^{2}\left( \frac{1}{4}\right) , \\
F\left( \frac{1}{2},\frac{1}{2};2;\frac{1}{2}\right) &=&\frac{4}{\pi ^{3/2}}%
\Gamma ^{2}\left( \frac{3}{4}\right) .
\end{eqnarray*}%
Then%
\begin{eqnarray*}
-\mathcal{G}_{4}^{\prime }\left( \frac{1}{2}\right) &=&\left[ \frac{2}{%
\left( 1-x\right) ^{3/2}}-\frac{\pi }{8\left( 1-x\right) }F\left( \frac{1}{2}%
,\frac{1}{2};2;x\right) e^{\mathcal{K}\left( \sqrt{x}\right) }\right]
_{x=1/2} \\
&=&4\sqrt{2}-\frac{1}{\sqrt{\pi }}\Gamma ^{2}\left( \frac{3}{4}\right) \exp
\left( \frac{\Gamma ^{2}\left( 1/4\right) }{4\sqrt{\pi }}\right) .
\end{eqnarray*}%
Let $x=r^{2}$. Then (\ref{H<>}) implies (\ref{eK-eK'<>}). In a similar way,
we can prove that (\ref{eK-eK'<>}) reverses for $r\in \left( 1/\sqrt{2}%
,1\right) $.
\end{proof}

\section{Concluding remarks}

In this paper, we investigated the absolute monotonicity of the functions $%
\pm \mathcal{G}_{p}\left( x\right) $ and $\pm \mathcal{G}_{p}^{\left(
k\right) }\left( x\right) $ for $k\in \mathbb{N}$ on $\left( 0,1\right) $,
where 
\begin{equation*}
\mathcal{G}_{p}\left( x\right) =e^{\mathcal{K}\left( \sqrt{x}\right) }-\frac{%
p}{\sqrt{1-x}}=\sum_{n=0}^{\infty }b_{n}x^{n}-p\sum_{n=0}^{\infty
}W_{n}x^{n}.
\end{equation*}%
These absolute monotonicity results yield several new functional inequality
related to $\mathcal{K}\left( r\right) $. It is worth empathizing our
approach used in this paper is very novel and interesing. Follow the usual
method, ones should determine the signs of those coefficients of power
series of $\pm \mathcal{G}_{p}\left( x\right) $ and $\pm \mathcal{G}%
_{p}^{\left( k\right) }\left( x\right) $ by directly using the recurrence
formula (\ref{bn+1/bn>}) and $W_{n}$ defined by (\ref{Wn}), which is very
complicated and hard. By our new idea (Steps 1 and 2), we found a surprising
property about $b_{n}$ and $W_{n}$, that is, the sequence $\left\{
b_{n}/W_{n}\right\} _{n\geq 1}$ is increasing with $\lim_{n\rightarrow
\infty }\left( b_{n}/W_{n}\right) =4$. It is with this property that we can
conveniently prove the desired absolute monotonicity results. Such new idea
may contribute to deal with certain problems on monotonicity, convexity and
absolute monotonicity occured in special functions.

Finally, we present two remarks and a problem.

\begin{remark}
It was proved in \cite{Yang-JMI-15-2021} that%
\begin{equation*}
\mathcal{F}_{1/2}\left( x\right) =\sqrt{1-x}e^{\mathcal{K}\left( \sqrt{x}%
\right) }=\sum_{n=0}^{\infty }a_{n}x^{n},
\end{equation*}%
where $a_{0}=e^{\pi /2}$ and $a_{n}<0$ for all $n\geq 1$. Differentiation
yields%
\begin{equation*}
\sqrt{1-x}\mathcal{F}_{1/2}^{\prime }\left( x\right) =\left[ \frac{\pi }{8}%
F\left( \frac{1}{2},\frac{1}{2};2;x\right) -\frac{1}{2}\right] e^{\mathcal{K}%
\left( \sqrt{x}\right) }=G\left( x\right) .
\end{equation*}%
Can prove that $G^{\prime }\left( x\right) $ is absolutely monotonic on $%
\left( 0,1\right) $ by the absolute monotonicity of $-\mathcal{F}%
_{1/2}^{\prime }\left( x\right) $? Or vice versa?
\end{remark}

\begin{remark}
Since $b_{0}/W_{0}>b_{1}/W_{1}$, we see that the sequence $\left\{
b_{n}/W_{n}\right\} _{n\geq 0}$ is decreasing for $0\leq n\leq 1$ and
increasing for $n\geq 1$. Making use of the unimodal type monotonicity rule
of the ratio of power series given in \cite[Theorem 2.1]{Yang-JMAA-428-2015}%
, we can only deduce that%
\begin{equation*}
\mathcal{F}_{1/2}\left( x\right) =\frac{e^{\mathcal{K}\left( \sqrt{x}\right)
}}{1/\sqrt{1-x}}=\frac{\sum_{n=0}^{\infty }b_{n}x^{n}}{\sum_{n=0}^{\infty
}W_{n}x^{n}}
\end{equation*}%
is decreasing on $\left( 0,1\right) $. Employing the recurrence technique,
however, Yang and Tian have shown that $-\mathcal{F}_{1/2}^{\prime }\left(
x\right) $ is absolutely monotonic on $\left( 0,1\right) $. In a similar
way, by the monotonicity rule of the ratio of power series given in \cite%
{Biernacki-AUMC-S-9-1955}, we can draw a conclusion that%
\begin{equation*}
x\mapsto \mathcal{J}\left( x\right) =\frac{e^{\mathcal{K}\left( \sqrt{x}%
\right) }-e^{\pi /2}}{1/\sqrt{1-x}-1}=\frac{\sum_{n=1}^{\infty }b_{n}x^{n}}{%
\sum_{n=1}^{\infty }W_{n}x^{n}}
\end{equation*}%
is increasing on $\left( 0,1\right) $. However, by the increasing property
of $\left\{ b_{n}/W_{n}\right\} _{n\geq 1}$ and recurrence technique, we can
prove that this function $\mathcal{J}\left( x\right) $ is absolutely
monotonic on $\left( 0,1\right) $, and the proof of which is left to the
reader.
\end{remark}

\begin{remark}
So far, several interesting questions remain unresolved involving the
absolutely monotonic functions related to the asymptotic formula (\ref{Kaf}%
), for example,

(i) In \cite[Conjecture 1]{Tian-RM-77-2022}, it was conjectured that the
function $-f_{4}^{\prime \prime \prime }\left( x\right) $ is absolutely
monotonic on $\left( 0,1\right) $, where $f_{p}\left( x\right) $ is defined
in Introduction of this paper.

(ii) In \cite[Conjecture 1]{Tian-RM-77-2022}, it was conjectured that the
function $1/f_{4}\left( x\right) $ is absolutely monotonic on $\left(
0,1\right) $.

(iii) Is the function $\left( f_{4}\left( x\right) -1\right) /\left(
1-x\right) $ absolutely monotonic on $\left( 0,1\right) $? This is a special
of the open problem proposed in \cite[Appendix G, Open Problems (2)]%
{Anderson-CIIQM-JW-1997}, which is still open.
\end{remark}

\noindent \textbf{Acknowledgements}

\noindent{We would like to thank the three referees for the careful review
on the manuscript. This research was supported by the National Natural
Science Foundation of China (11971142) and the Natural Science Foundation of
Zhejiang Province (LY24A010011).}

\noindent\textbf{Data availability}\newline
Data sharing not applicable to this article as no datasets were generated or
analysed during the current study.

\noindent\textbf{Declarations}

\noindent \textbf{Conflict of interest}\ \ The authors declare that they
have no conflict of interest.


\begin{thebibliography}{99}
\bibitem{Byrd-HEIES-1971} P. F. Byrd and M. D. Friedman, \emph{Handbook of
Elliptic Integrals for Engineers and Scientists}, Springer-Verlag, New York,
1971.

\bibitem{Qiu-SIAM-JMA-27-1996} S.-L. Qiu and M. K. Vamanamurthy, Sharp
estimates for complete elliptic integrals, \emph{SIAM J. Math. Anal.} 
\textbf{27} (1996), no. 3, 823--834. https://doi.org/10.1137/052704

\bibitem{Yang-MIA-21-2018} Z.-H. Yang, W.-M. Qian and Y.-M, Chu,
Monotonicity properties and bounds involving the complete elliptic integrals
of the first kind, \emph{Math. Inequal. Appl.} \textbf{21} (2018), no. 4,
1185-1199. doi:10.7153/mia-2018-21-82

\bibitem{Yang-JMI-15-2021} Z.-H. Yang and J. Tian, Absolutely monotonic
functions involving the complete elliptic integrals of the first kind with
applications, \emph{J. Math. Inequal.} \textbf{15} (2021), no. 3,
1299--1310. DOI: 10.7153/jmi-2021-15-87

\bibitem{Anderson-SIAM-JMA-21-1990} G. D. Anderson, M. K. Vamanamurthy and
M. Vuorinen, Functional inequalities for complete elliptic integrals and
their ratios, \emph{SIAM J. Math. Anal.} \textbf{21} (1990), no. 2,
536--549. DOI: 10.1137/0521029

\bibitem{Qiu-PMJ-5-1995} S.-L Qiu and M. K Vamanamurthy, Elliptic integrals
and the modulus of Gr\"{o}tzsch ring, \emph{Panamer. Math. J.} \textbf{5}
(1995), 41--60.

\bibitem{Yang-AADM-13-2019} Z.-H. Yang and J.-F. Tian, Convexity and
monotonicity for elliptic integrals of the first kind and applications, 
\emph{Appl. Anal. Discrete Math.} \textbf{13} (2019), 240--260.
https://doi.org/10.2298/AADM171015001Y

\bibitem{Tian-RM-77-2022} J.-F. Tian and Z.-H. Yang, Several absolutely
monotonic functions related to the complete elliptic integral of the first
kind, \emph{Results Math.} \textbf{77} (2022), Art. no. 109.
https://doi.org/10.1007/s00025-022-01641-4

\bibitem{Alzer-ITSF-31-2020} H. Alzer and K. C. Richards, A concavity
property of the complete elliptic integral of the first kind, \emph{Integral
Transforms Spec. Funct.} \textbf{31} (2020), no. 9, 758--768.
https://doi.org/10.1080/10652469.2020.1738423

\bibitem{Tian-RM-78-2023} J.-F. Tian and Z.-H. Yang, Convexity and
monotonicity involving the complete elliptic integral of the first kind, 
\emph{Results Math. }\textbf{78} (2023), Art. no. 29.
https://doi.org/10.1007/s00025-022-01799-x

\bibitem{Wang-AADM-14-2020} M.-K. Wang, H.-H. Chu, Y.-M. Li, and Y.-M. Chu,
Answers to three conjectures on convexity of three functions involving
complete elliptic integrals of the first kind, \emph{Appl. Anal. Discrete
Math.} \textbf{14} (2020), 255--271. https://doi.org/10.2298/AADM190924020W

\bibitem{Yang-AMS-42-2022} Z. Yang and J. Tian, Absolute monotonicity
involving the complete elliptic integrals of the first kind with
applications, \emph{Acta Math. Sci.} \textbf{42} (2022), 847--864.
https://doi.org/10.1007/s10473-022-0302-x

\bibitem{Qiu-SIAM-JMA-30-1999} S.-L. Qiu, M. Vuorinen, Infinite products and
the normalized quotients of hypergeometric functions, \emph{SIAM J. Math.
Anal.} \textbf{30} (1999), no. 5, 1057--1075.
https://doi.org/10.1137/S003614109732680

\bibitem{Yang-MIA-20-2017} Z.-H. Yang, Y.-M. Chu, A monotonicity property
involving the generalized elliptic integrals of the first kind, \emph{Math.
Inequal. Appl.} \textbf{20} (2017), no. 3, 729--735. doi:10.7153/mia-20-46

\bibitem{Wang-AMS-39B-2019} M.-K. Wang, W. Zhang and Y.-M. Chu,
Monotonicity, convexity and inequalities involving the generalized elliptic
integrals, \emph{Acta Math. Sci.} \textbf{39B} (2019), no. 5, 1440--1450.
https://doi.org/10.1007/s10473-019-0520-z

\bibitem{Yang-RJ-48-2019} Z.-H. Yang and J.-F. Tian, Sharp inequalities for
the generalized elliptic integrals of the first kind, \emph{Ramanujan J.} 
\textbf{48} (2019), 91--116. https://doi.org/10.1007/s11139-018-0061-4

\bibitem{Richards-ITSF-32-2020} K. C. Richards and J. N. Smith, A concavity
property of generalized complete elliptic integrals, \emph{Integral
Transforms Spec. Funct.} \textbf{32} (2020), no. 3, 1--13. DOI:
10.1080/10652469.2020.1815726

\bibitem{Zhao-RACSAM-115-2021} T.-H. Zhao, M.-K. Wang and Y.-M. Chu,
Monotonicity and convexity involving generalized elliptic integral of the
first kind, \emph{Rev. Real Acad. Cienc. Exactas Fis. Nat. Ser. A-Mat. }%
\textbf{115}\emph{\ }(2021): 46. https://doi.org/10.1007/s13398-020-00992-3

\bibitem{Zhao-JMI-15-2021} T.-H. Zhao, M.-K. Wang and Y.-M. Chu, Concavity
and bounds involving generalized elliptic integral of the first kind, \emph{%
J. Math. Inequal.} \textbf{15} (2021), no. 2, 701--724.
doi:10.7153/jmi-2021-15-50

\bibitem{Chen-RACSAM-116-2022} Y.-J. Chen and T.-H. Zhao, On the
monotonicity and convexity for generalized elliptic integral of the first
kind, \emph{Rev. Real Acad. Cienc. Exactas Fis. Nat. Ser. A-Mat.} \textbf{116%
} (2022): 77. https://doi.org/10.1007/s13398-022-01211-x

\bibitem{Wang-IJPAM-54-2023} M.-K. Wang, T.-H. Zhao, X.-J. Ren, Y.-M. Chu
and Z.-Y. He, Monotonicity and concavity properties of the Gaussian
hypergeometric functions, with applications, \emph{Indian J. Pure Appl. Math.%
} \textbf{54} (2023), 1105--1124. https://doi.org/10.1007/s13226-022-00325-7

\bibitem{Carlson-SIAM-JMA-16-1985} B. C. Carlson, J. L. Guatafson,
Asymptotic expansion of the first elliptic integral, \emph{SIAM J. Math.
Anal.} \textbf{16} (1985), no. 5, 1072--1092. https://doi.org/10.1137/0516080

\bibitem{Anderson-SIAM-JMA-23-1992} G. D. Anderson, M. K. Vamanamurthy and
M. Vuorinen, Functional inequalities for hypergeometric functions and
complete elliptic integrals, \emph{SIAM J. Math. Anal.} \textbf{23} (1992),
no. 2, 512--524. DOI 10.1137/0523025

\bibitem{Qiu-SIAM-JMA-29-1998} S.-L. Qiu, M. K. Vamanamurthy and M.
Vuorinen, Some inequalities for the growth of elliptic integrals, \emph{SIAM
J. Math. Anal.} \textbf{29} (1998), no. 5, 1224--1237.
https://doi.org/10.1137/S0036141096310491

\bibitem{Alzer-MPCPS-124(2)-1998} H. Alzer, Sharp inequalities for the
complete elliptic integral of the first kind, \emph{Math. Proc. Camb. Phil.
Soc.} \textbf{124} (1998), no. 2, 309--314. DOI:
https://doi.org/10.1017/S0305004198002692

\bibitem{Alzer-JCAM-172-2004} H. Alzer and S.-L. Qiu, Monotonicity theorems
and inequalities for the complete elliptic integrals, \emph{J. Comput, Appl.
Math.} \textbf{172} (2004), 289--312.
https://doi.org/10.1016/j.cam.2004.02.009

\bibitem{Yang-JIA-2014-192} Z.-H. Yang, Y.-Q. Song and Y.-M. Chu, Sharp
bounds for the arithmetic-geometric mean, \emph{J. Inequal. Appl.} \textbf{%
2014} (2014): 192. DOI: 10.1186/1029-242X-2014-192

\bibitem{Yang-MIA-23-2020} Z.-H. Yang, W.-M. Qian, W. Zhang and Y.-M. Chu,
Notes on the complete elliptic integral of the first kind, \emph{Math.
Inequal. Appl.} \textbf{23} (2020), no. 1, 77--93. DOI:
10.7153/mia-2020-23-07

\bibitem{Yang-RACSAM-115-2021} Z.-H. Yang, J.-F. Tian and Y.-R. Zhu, A sharp
lower bound for the complete elliptic integrals of the first kind, \emph{%
Rev. Real Acad. Cienc. Exactas Fis. Nat. Ser. A-Mat.} \textbf{115} (2021):
8. https://doi.org/10.1007/s13398-020-00949-6

\bibitem{Petrovic-PMUB-1-1932} M. Petrovi\'{c}, Sur une \'{e}quation
fonctionnelle, \emph{Publ. Math. Univ. Belgrade} \textbf{1} (1932), 149--156.

\bibitem{Abramowitz-HMFFGMT-1972} M. Abramowitz and I. A. Stegun (Eds), 
\emph{Handbook of Mathematical Functions with Formulas, Graphs, and
Mathematical Tables}, National Bureau of Standards, Applied Mathematics
Series 55, 10th printing, Dover Publications, New York and Washington, 1972.

\bibitem{Yang-JMAA-428-2015} Z.-H. Yang, Y.-M. Chu, and M.-K. Wang,
Monotonicity criterion for the quotient of power series with applications, 
\emph{J. Math. Anal. Appl.} \textbf{428} (2015), no. 1, 587--604.
https://doi.org/10.1016/j.jmaa.2015.03.043

\bibitem{Biernacki-AUMC-S-9-1955} M. Biernacki, J. Krzyz, On the
monotonicity of certain functionals in the theory of analytic functions, 
\emph{Annales Universitatis Mariae Curie-Sklodowska}, \textbf{9} (1955),
135--147.

\bibitem{Anderson-CIIQM-JW-1997} G. D Anderson, M. K. Vamanamurthy and M.
Vuorinen, \emph{Conformal Invariants, Inequalities, and Quasiconformal Maps}%
, John Wiley \& Sons, New York, 1997.
\end{thebibliography}
\end{document}